\documentclass[12pt,a4paper,reqno]{amsart}
\usepackage{graphicx}
\usepackage[all]{xy}
\usepackage{amssymb}
\usepackage{amsmath}
\usepackage[utf8]{inputenc}
\usepackage{amsthm}

\usepackage[backref=page]{hyperref}
\renewcommand*{\backref}[1]{}
\renewcommand*{\backrefalt}[4]{%
    \ifcase #1 (Not cited.)%
    \or        (Cited on page~#2.)%
    \else      (Cited on pages~#2.)%
    \fi}

\newtheorem{proposition}{Proposition}

\newtheorem{theorem}{Theorem}

\begin{document}

\title[What is  an internal groupoid?]{What is  an internal groupoid?}


\author{N. Martins-Ferreira}
\address[Nelson Martins-Ferreira]{Instituto Politécnico de Leiria, Leiria, Portugal}
\thanks{ }
\email{martins.ferreira@ipleiria.pt}

\begin{abstract}
An answer to the question investigated in this paper brings a new characterization of internal groupoids such that: (a) it holds even when finite limits are not assumed to exist; (b) it is a full subcategory of the category of invo\-lutive-2-links, that is, a category whose objects are morphisms equipped with a pair of interlinked involutions. This result highlights the fact that even thought internal groupoids are internal categories equipped with an involution, they can equivalently be seen as tri-graphs with an involution. Moreover, the structure of a tri-graph with an involution can be further contracted into a simpler structure consisting of one morphism with two interlinked involutions. This approach highly contrasts with the one where groupoids are seen as reflexive graphs on which a multiplicative structure is defined with inverses.



\end{abstract}

\maketitle

\date{Received: date / Accepted: date}

\section{Introduction}
\label{intro}

It is needless to say that internal groupoids are an important class of internal categories which have been thoroughly investigated on a diversity of subtopics in category theory with relevant applications to other areas such as in particular algebra and geometry. A celebrated paper  by D. Bourn~\cite{Bourn91} has certainly contributed  to improve and enrich the field on the algebraic side with the subsequent aid of many researchers who became interested in the remarkable properties of the fibration of points and its classifying properties (\cite{Bourn13} see also \cite{BB}). On the side of geometry, the book \emph{Topology and Groupoids} by R. Brown has been influential in continuing the ideas and work of Grothendieck \cite{Grothendieck} and Ehresmann \cite{Ehresmann} on differentiable groupoids and on the fundamental groupoid of a space \cite{Brown1,Brown}. Although an internal groupoid is an instance of an internal category, Brandt \cite{Brandt} predates Eilenberg and Mac Lane~\cite{Eilenberg-MacLane} in delineating an axiomatic portrait of a (connected) groupoid (\cite{Voight} Remark 19.3.12). Lie groupoids are internal groupoids in the category of smooth manifolds while topological groupoids are internal groupoids in the category of topological spaces. Thus, groupoids are a crossroad between algebra and geometry by virtue of category theory. For a survey see e.g. \cite{Higgins,Weinstein}.  

It is well known that on some grounds the structure of an internal groupoid simplifies itself due to a coincidence between property and structure. For example, groupoids internal to the category of abelian groups (or any abelian category) are nothing but group homomorphisms whereas groupoids internal to the category of groups (or any semiabelian category) are the same as internal crossed-modules~\cite{Janelidze}. One  draw-back of working with a category of internal groupoids is the complexity of its plain structure. It involves a reflexive graph, together with a morphism expressing multiplication (or composition) which requires an appropriate pullback as domain (making composable pairs of arrows meaningful), it requires an involution providing inverses, as well as several suitable compatibility conditions, not to mention the iterated pullback necessary to express associativity (Section \ref{sec: groupoids}). For that reason various authors have restricted the study of internal groupoids to situations in which the plain structure is significantly reduced due to good properties displayed by the ambient category. That is the case for example of Mal'tsev and weakly Mal'tsev categories in which the category of internal groupoids is a full subcategory of the category of reflexive graphs (see e.g. \cite{NMF.20a}, see also \cite{NMF.17}, Section 3). Although the algebraically flavoured approach (in which a groupoid is seen as a set equipped with a partial multiplication) is not satisfactory from the point of view of internal structures, it does suggest a simplification on the structure as proposed by Brandt \cite{Brandt}. 

The purpose of this paper is to show that independently of its ambient, the category of internal groupoids is always a full subcategory of the category whose objects are morphisms equipped with a pair of interlinked involutions. This category, in some respects, is even simpler than the category of reflexive graphs. It will be called the category of involutive-2-links (Section \ref{sec: inv2links}).

Throughout this paper we work on an ambient category in which no limits nor colimits are assumed to exist (a good example to bear in mind is the cateogry of smooth manifolds). In particular this means that pullbacks and pushouts are to be understood as properties of commutative squares. When a square is at the same time a pullback and a pushout we call it an \emph{exact square} (such squares are also called bicartesian squares, Dolittle diagrams or pulation squares \cite{Ban}). A span which can be completed into an exact square is called an exact span. In the same way, a cospan which can be completed into an exact square is called an exact cospan. A digraph (i.e. a pair of parallel morphisms) which is at the same time an exact span and an exact cospan is called bi-exact. This notion will be needed in the main result (Theorem \ref{thm: main}).



\section{The structure of an internal groupoid}\label{sec: groupoids}

An internal groupoid can be presented as a structure of the form 
\begin{equation}\label{diag: internal groupoid}\xymatrix{C_2 \ar@<1.5ex>[r]^{\pi_2}\ar@<-1.5ex>[r]_{\pi_1}\ar[r]|{m} & C_1 \ar@(ur,ul)[]_{i}  \ar@<1ex>[r]^{d} \ar@<-1ex>[r]_{c} & C_0 \ar[l]|{e} }
\end{equation}
such that
\begin{eqnarray}
de = 1_{C_1} =ce \\
d m = d \pi_2,\quad c m = c \pi_1,\quad d\pi_1=c\pi_2\\
di=c,\quad ci=d, \quad i^2=1_{C_1},\quad ie=e
\end{eqnarray}
and satisfying the following further properties:
\begin{enumerate}
\item the commutative square 
\begin{equation}\label{diag: d c square}
\vcenter{
\xymatrix{C_2 \ar[r]^{\pi_2}\ar[d]_{\pi_1} & C_1 \ar[d]^{c}\\
 C_1\ar[r]^{d} & C_0}}
\end{equation}
is a pullback square;
\item $m\langle 1_{C_1},ed\rangle=1_{C_1}=m\langle ec,1_{C_1}\rangle$;
\item $m\langle 1_{C_1},i\rangle=ec,\quad m\langle i,1_{C_1}\rangle=ed$;
\item the cospan $\xymatrix{C_2\ar[r]^{d\pi_2} & C_0 & C_1 \ar[l]_{c}}$ can be completed into a pullback square
\begin{equation}\label{diag: dpi_2 c square}
\vcenter{
\xymatrix{C_3 \ar[r]^{p_2}\ar[d]_{p_1} & C_1 \ar[d]^{c}\\
 C_2\ar[r]^{d\pi_2} & C_0}}
\end{equation}

\item $m(1\times m)=m(m\times 1)$, where $(1\times m),(m\times 1)\colon C_3\to C_2$ are morphisms uniquely determined as
\begin{align*}
    \pi_2(m\times 1)=p_2\\
    \pi_1(m\times 1)=m\\
    \pi_2(1\times m)=m\langle \pi_2p_1,p_2 \rangle\\
    \pi_1(1\times m)=\pi_2.
\end{align*}
\end{enumerate}

There is, clearly, a redundancy in this structure and its properties. For example, as soon as the commutative square $d\pi_1=c\pi_2$, displayed as  $(\ref{diag: d c square})$, is required to be a pullback, the morphisms $\pi_1$ and $\pi_2$ are uniquely determined up to isomorphism. A remarkable observation due to D.~Bourn~\cite{Bourn87} shows that the existence of a morphism $i\colon C_1\to C_1$ such that $di=c$, $i^2=1_{C_1}$, $ie=e$ and $m\langle 1_{C_1},i\rangle=ec$, $m\langle i,1_{C_1}\rangle=ed$ can be interpreted as a property. Indeed, the existence of the involution morphism $i$ is equivalent to the property that the commutative square $dm=d\pi_2$ is a pullback. This is of course consistent with the general fact that inverses, when they exist, are uniquely determined. However, it raises the question of finding a minimal structure that can carry enough information to encode the notion of internal groupoid.

With a few exceptions (e.g. \cite{BournJanel98,Janelidze91}), most papers so far  have considered the case in which pullbacks are available as a canonical construction and describe an internal groupoid as a structure of the form 
\begin{equation}\label{diag: multiplicative graph}\xymatrix{C_2 \ar[r]^{m} & C_1 \ar@<1ex>[r]^{d} \ar@<-1ex>[r]_{c} & C_0 \ar[l]|{e} }
\end{equation}
in which the morphisms $\pi_1$ and $\pi_2$ are canonically obtained by requiring the object $C_2$ to be the result of taking the pullback of $c$ along $d$ as displayed in $(\ref{diag: d c square})$.  If the involution morphism $i$ is explicitly needed then it is recovered from the pullback square $dm=d\pi_2$.

 Much work has been done in studying internal groupoids focused on their underlying reflexive graphs. This is because in some algebraic contexts such as Mal'tsev or weakly Mal'tsev categories \cite{NMF.20a}  the morphism $m$ is uniquely determined provided it exists.

The purpose of this paper is to show that a simpler description is equally possible in general categories if we shift the attention from the underlying reflexive graph to the underlying multiplicative structure.

We will show that the relevant information is enclosed in the tri-graph with involution 
\begin{equation}\label{diag: trigraph with involution}\xymatrix{C_2 \ar@<1.5ex>[r]^{\pi_2}\ar@<-1.5ex>[r]_{\pi_1}\ar[r]|{m} & C_1 \ar@(ur,rd)[]^{i} }
\end{equation}
from which the square $(\ref{diag: d c square})$ is obtained by taking the pushout of $\pi_1$ and $\pi_2$. In fact, we will see that this tri-graph can be further contracted into a single morphism $m\colon C_2\to C_1$ equipped with a pair of involutions $\theta,\varphi\colon C_2\to C_2$ satisfying the  condition $\theta\varphi\theta=\varphi\theta\varphi$. Note that the subgroup generated by $\theta$ and $\varphi$ is the Diedral group of order 6.

The structure of a \emph{link} \cite{SI7.2} was introduced to model planar curves and it consists of a morphism $f\colon A\to B$ equipped with an endomorphism  $\alpha\colon A\to A$ with no further conditions (even tough $\alpha$ is often required to be an isomorphism). When $\alpha$ is an involution we refer to it as an \emph{involutive-link}. For example, the domain morphism $d\colon C_1\to C_0$ together with the involution morphism $i$ is an involutive-link. It is thus appropriate to speak of an \emph{involutive-2-link} when we have a pair of involutions $(\theta,\varphi)$ which are interlinked by the formula $\theta\varphi\theta=\varphi\theta\varphi$. The two involutions $\theta$ and $\varphi$ that we will consider are determined (in a unique way) from the structure  displayed in $(\ref{diag: trigraph with involution})$ by equations $(\ref{eq: theta and varphi 1})$--$(\ref{eq: theta and varphi 3})$ below.
If we write $m(x,y)$ as $xy$ and $i(x)$ as $x^{-1}$ then $\theta(x,y)=(x^{-1},xy)$ and $\varphi(x,y)=(xy,y^{-1})$, assuming of course that $C_2$ is the pullback of $d$ and $c$ as illustrated in the structure  of an internal groupoid displayed in $(\ref{diag: internal groupoid})$. A list of examples is provided in Section \ref{sec: examples}.

The two main observations of this paper are (a) the category of internal groupoids is a full subcategory of the category of involutive-2-links and (b) an involutive-2-link is an internal groupoid if and only if the conditions of Theorem \ref{thm: main} are satisfied. Moreover, straightforward generalizations of these results are expected in higher dimensional categorical structures such as Loday's cat-$n$-groups \cite{Loday}. Furthermore, applications to the categorical Galois theory of G. Janelidze \cite{Janelidze91} are expected as well.  

\section{The category of involutive-2-links}\label{sec: inv2links}

The category of involutive-2-links is the category whose objects are triples $(\theta,\varphi,m)$ in which $m\colon A\to B$ is a morphism while $\theta,\varphi\colon A\to A$ are involutions, i.e., $\theta^2=\varphi^2=1_A$, that satisfy the interlinked condition  $\theta\varphi\theta=\varphi\theta\varphi$. Moreover, for convenience, we will restrict our attention to the triples $(\theta,\varphi,m)$ with the property that the three parallel morphisms $m$, $m\theta$ and $m\varphi$ are jointly monomorphic. This means that any morphism with codomain $A$ is uniquely determined as soon as the result of precomposing it with  $m$, $m\theta$ and $m\varphi$ is known. This is clearly the case if at least one of the three pairs of morphisms $(m,m\theta)$, $(m,m\varphi)$ or $(m\varphi,m\theta)$ are obtained as pullback squares, which will be the case. Indeed, as it will be apparent later on, the triple $(\theta,\varphi,m)$ is used to recover the tri-graph displayed in  (\ref{diag: trigraph with involution}) by setting $\pi_1=m\varphi$ and $\pi_2=m\theta$.

A morphism in the category of involutive-2-links, say from a triple $(\theta,\varphi,m\colon A\to B)$ to a triple $(\theta',\varphi',m'\colon A'\to B')$, consists of a single morphism $f\colon B\to B'$ with the property that there exists a morphism $\Bar{f}\colon A\to A'$ such that 
\begin{align}
    m'\Bar{f}=fm\label{eq: conditions on f bar 1}\\
    \theta'\Bar{f}=\Bar{f}\theta \label{eq: conditions on f bar 2}\\
    \varphi'\Bar{f}=\Bar{f}\varphi.\label{eq: conditions on f bar 3}
\end{align}

It is worth noting that the requirement on the triple $(\theta',\varphi',m')$ asserting that the three parallel morphisms
\begin{equation}\label{diag: three parallel morphisms}\xymatrix{A' \ar@<1.5ex>[r]^{m'\theta'}\ar@<-1.5ex>[r]_{m'\varphi'}\ar[r]|{m'} & B' }
\end{equation}
are jointly monomorphic has the effect of uniquely determining the morphism $\Bar{f}$ as
\begin{align*}
    m'\Bar{f}=fm\\
    m'\theta'\Bar{f}=fm\theta\\
    m'\varphi'\Bar{f}=fm\varphi.
\end{align*}

This restriction on involutive-2-links (of being jointly monomorphic) gives rise to a fully faithful functor that associates a triple $(\theta,\varphi,m)$ to every internal groupoid such as the one displayed in $(\ref{diag: trigraph with involution})$, according to the formulas:
\begin{align}
    m\varphi=\pi_1,\quad m\theta=\pi_2\label{eq: theta and varphi 1}\\
    \pi_1\varphi= m,\quad \pi_1\theta=i\pi_1\label{eq: theta and varphi 2}\\
    \pi_2\varphi=i\pi_2,\quad \pi_2\theta= m.\label{eq: theta and varphi 3}
\end{align}
Let us observe that even when $i$ is not made explicit, the two involutions $\theta$ and $\varphi$ are still uniquely determined because the object  $C_2$ is not only the pullback of $d$ and $c$ but also the kernel pair of $d$ and the kernel pair of $c$ and hence both pairs $(m,\pi_1)$ and $(m,\pi_2)$ are in particular jointly monomorphic.

\begin{proposition}\label{thm: the functor}
There exists a fully faithful functor from the category of internal groupoids to the category of involutive-2-links. The functor takes an internal groupoid such as the one displayed in $(\ref{diag: internal groupoid})$ and converts it into the triple $(\theta,\varphi,m)$ in which $m\colon C_2\to C_1$ is the morphism carrying the multiplicative structure of the groupoid while the endomorphisms $\theta,\varphi\colon C_2\to C_2$ are determined as $\theta=\langle i\pi_1,m\rangle$ and $\varphi=\langle m,i\pi_2\rangle$. Taking into account that $C_2$ is the pullback of $d$ and $c$. 
\end{proposition}
\begin{proof}
In some sense the functor keeps the underlying multiplicative structuring morphism $m\colon C_2\to C_1$ and contracts the relevant information displayed in $(\ref{diag: trigraph with involution})$ as two endomorphisms, $\theta$ and $\varphi$. It is thus clear that the tri-graph
\begin{equation}\label{diag: tri-graph m etc}\xymatrix{C_2 \ar@<1.5ex>[r]^{m\theta=\pi_2}\ar@<-1.5ex>[r]_{m\varphi=\pi_1}\ar[r]|{m} & C_1 }
\end{equation}
is jointly monomorphic. In fact it is much more, as $(\pi_1,\pi_2)$ is obtained by pullback, while $(m,\pi_2)$ and $(\pi_1,m)$ are the kernel pairs of the morphisms $d$ and $c$, respectively. It is routine to check that $\theta$ and $\varphi$ are involutions and that the interlinked condition  $\theta\varphi\theta=\varphi\theta\varphi$ is satisfied. This also explains why the functor is faithful and well defined on morphisms.
To prove that it is full, let us consider two internal groupoids with the relevant tri-graph structure and involution as illustrated in diagram  $(\ref{diag: trigraph with involution})$, together with a morphism $f\colon C_1\to C'_1$ such that there exists a (unique) $\Bar{f}\colon C_2\to C'_2$ 
satisfying the three conditions $(\ref{eq: conditions on f bar 1})$--$(\ref{eq: conditions on f bar 3})$ with $\theta,\theta',\varphi,\varphi'$ obtained as in $(\ref{eq: theta and varphi 1})$--$(\ref{eq: theta and varphi 3})$, as illustrated.
\begin{equation}\label{diag: tri-graph f}\xymatrix{C_2\ar@{-->}[d]_{\bar{f}} \ar@<1.5ex>[r]^{m\theta=\pi_2}\ar@<-1.5ex>[r]_{m\varphi=\pi_1}\ar[r]|{m} & C_1\ar[d]^{f} \ar@(ur,rd)[]^{i}\\
C'_2 \ar@<1.5ex>[r]^{m'\theta'=\pi'_2}\ar@<-1.5ex>[r]_{m'\varphi'=\pi'_1}\ar[r]|{m'} & C'_1 \ar@(ur,rd)[]^{i'}}
\end{equation}
Under these conditions, it follows that  $i'f=fi$ which is then used in collaboration with the fact that $C_2$ and $C'_2$ are pullbacks to assert that 
\begin{align}
    \bar{f}\langle 1,ed \rangle=\langle 1,ed \rangle f\label{eq: e1 cartesian}\\
    \bar{f}\langle ec,1 \rangle=\langle ec,1 \rangle f\label{eq: e2 cartesian}
\end{align}
from which we obtain the complete diagram 
\begin{equation}\label{diag: tri-graph f complete}\xymatrix{C_2\ar@{-->}[d]_{\bar{f}} \ar@<1.5ex>[r]^{\pi_2}\ar@<-1.5ex>[r]_{\pi_1}\ar[r]|{m} & C_1\ar[d]^{f}  \ar@(ur,ul)[]_{i} 
 \ar@<1ex>[r]^{d} \ar@<-1ex>[r]_{c} & C_0 \ar[l]|{e} \ar@{-->}[d]^{f_0}
\\
C'_2 \ar@<1.5ex>[r]^{\pi'_2}\ar@<-1.5ex>[r]_{\pi'_1}\ar[r]|{m'} & C'_1  \ar@(dr,dl)[]_{i'} \ar@<1ex>[r]^{d'} \ar@<-1ex>[r]_{c'} & C'_0 \ar[l]|{e'} }
\end{equation}
with $f_0$ being uniquely determined by the property that $f_0 d=d'f$ and $f_0 c=c' f$. Moreover, $fe=e'f_0$ follows from one of the assertions $(\ref{eq: e1 cartesian})$ or $(\ref{eq: e2 cartesian})$.
\end{proof}

\section{Examples}\label{sec: examples}

The examples are provided with the notation of sets and maps but are supposed to be defined in any category where the required constructions are possible. In each case we provide a set $C_1$, to be interpreted as the set of arrows, a set $C_2$, interpreted as the set of composable arrows, one function $m\colon C_2\to C_1$ and two involutions $\theta,\varphi\colon C_2\to C_2$ such that $\theta\varphi\theta=\varphi\theta\varphi$. In addition, the tri-graph
\begin{equation}\label{diag: tri-graph monic}\xymatrix{C_2 \ar@<1.5ex>[r]^{m\theta}\ar@<-1.5ex>[r]_{m\varphi}\ar[r]|{m} & C_1 }
\end{equation}
is  a ternary relation.

Let $X$ be a set, $G=(G,\cdot,1,()^{-1})$ a group, $M=(M,\cdot,1)$ a monoid, $S=(S,\cdot,()^{-1})$ and inverse semigroup and $U=(U_i)_{i\in I}$ an open cover of a smooth manifold:
\begin{enumerate}
    \item The set $X$, if considered as a discrete groupoid, becomes an involutive-2-link as:
    \begin{enumerate}
        \item $C_1=X$, $C_2=X$, $m(x)=x$;
        \item $\theta(x)=x$, $\varphi(x)=x$.
    \end{enumerate}
    \item Whereas if the set $X$ is considered as a co-discrete groupoid, it becomes the involutive-2-link:
    \begin{enumerate}
        \item $C_1=X\times X$, $C_2=X\times X \times X$;
        \item $m(x,y,z)=(x,z)$;
        \item $\theta(x,y,z)=(y,x,z)$;
        \item $\varphi(x,y,z)=(x,z,y)$.
    \end{enumerate}
    \item If an equivalence relation $R\subseteq X\times X$ is interpreted as a groupoid then it becomes an involutive-2-link as:
    \begin{enumerate}
        \item $C_1=R$, $C_2=\{(x,y,z)\in X\times X \times X\mid xRyRz\}$;
        \item $m(x,y,z)=(x,z)$;
        \item $\theta(x,y,z)=(y,x,z)$;
        \item $\varphi(x,y,z)=(x,z,y)$.
    \end{enumerate}
    \item The \v{C}ech  groupoid associated to the open cover $U$ is obtained in the form of an involutive-2-link as:
    \begin{enumerate}
        \item $C_1=\bigsqcup U_{ij}$, $C_2=\bigsqcup U_{ijk}$
        \item $m(x_{ijk})=x_{ik}$;
        \item $\theta(x_{ijk})=x_{jik}$;
        \item $\varphi(x_{ijk})=x_{ikj}$.
    \end{enumerate}
    \item The group $G$, considered as a one object groupoid gives rise to an involutive-2-link of the form:
    \begin{enumerate}
        \item $C_1=G$, $C_2=G \times G$;
        \item $m(x,y)=xy$;
        \item $\theta(x,y)=(x^{-1},xy)$;
        \item $\varphi(x,y)=(xy,y^{-1})$.
    \end{enumerate}
    \item If $\xi\colon G\times X\to X$ is a $G$-action on the set $X$, then it can be seen as an involutive-2-link as:
    \begin{enumerate}
        \item $C_1=G\times X$, $C_2=G \times G\times X$;
        \item $m(a,b,x)=(ab,x)$;
        \item $\theta(a,b,x)=(a^{-1},ab,x)$;
        \item $\varphi(a,b,x)=(ab,b^{-1},\xi(b,x))$.
    \end{enumerate}
    \item If $h\colon G\to M$ is a homomorphism from the group $G$ to the monoid $M$ then it can be seen as an involutive-2-link as:
     \begin{enumerate}
        \item $C_1=G\times M$, $C_2=G \times G\times M$;
        \item $m(a,b,c)=(ab,c)$;
        \item $\theta(a,b,c)=(a^{-1},ab,c)$;
        \item $\varphi(a,b,c)=(ab,b^{-1},h(b)+c)$.
    \end{enumerate}
    In addition, if $M$ acts on $G$ in the fashion of a crossed-module then the involutive-2-link is an internal structure in the category of monoids. Let us nevertheless observe that although the action $\xi$ on the structure of a crossed-module is necessary to equip the set $C_1=G\times M$ with the semidirect product $G\rtimes_{\xi} M$, the underlying groupoid is defined as soon as the homomorphism $h\colon G\to M$ is specified. 
    \item The inverse semigroup $S$, if interpreted as a groupoid in the form of an involutive-2-link, becomes:
    \begin{enumerate}
        \item $C_1=S$, $C_2=\{(x,y)\in S \times S\mid x^{-1}x=yy^{-1}\}$;
        \item $m(x,y)=xy$;
        \item $\theta(x,y)=(x^{-1},xy)$;
        \item $\varphi(x,y)=(xy,y^{-1})$.
    \end{enumerate}
    \item The simplest non-trivial involutive-2-link which is not obtained from a groupoid is:
 \begin{enumerate}
        \item $C_1=\{0,1\}$, $C_2=\{1,2,3\}$;
        \item $m:[1 2 3]\mapsto [0 1 0]$;
        \item $\theta:[1 2 3]\mapsto [213]$;
        \item $\varphi:[1 2 3]\mapsto [1 3 2]$.
    \end{enumerate}
    \item One last construction is obtained by combining several examples above. Let $B$ be a set  and consider two maps $g\colon S\to B$ and  $\phi\colon B\times X\to X$ (with $S$ and $X$ as above) such that
    \begin{align}
        \phi(g(s^{-1}s),x)=x=\phi(g(ss^{-1}),x)\\
        \phi(g(s's),x)=\phi(g(s'),\phi(g(s),x)).
    \end{align}
    Every subset $R\subseteq S\times X$ satisfying the following conditions
    \begin{enumerate}
        \item[(i)] if $(s,x)\in R$ then $(s^{-1}s,x)\in R$ and $(ss^{-1},\phi(g(s),x))\in R$
        \item[(ii)] if $(s,x)\in R$ and $(s',\phi(g(s),x))\in R$ then $(s's,x)\in R$
    \end{enumerate}
    gives rise to an involutive-2-link in which:
    \begin{enumerate}
    \item $C_1=R$ while  $C_2\subseteq S\times S\times X$ consists of those triples $(s',s,x)\in S\times S\times X$ such that  $s'^{-1}s'=ss^{-1}$,  $(s,x)\in R$ and $(s',\phi(g(s),x))\in R$
    \item $m(s',s,x)=(s's,x)$
    \item $\theta(s',s,x)=(s'^{-1},s's,x)$
    \item $\varphi(s',s,x)=(s's,s^{-1},\phi(g(s),x))$.
    \end{enumerate}
    If $B$ is a monoid and $g$ is a homomorphism then we recover example (7) by putting $X=B$ and $\phi(b,x)=b+x$. On the other hand, if $g$ is a bijection  then we recover example (6). Examples (5) and (8) are recovered as well.
    \end{enumerate}

It is worthwhile noting under which conditions an involutive  magma can be seen as an involutive-2-link.  

\begin{proposition}
    Let  $m\colon X\times X \to X$ be a magma structure on the set $X$ together with an involution $i\colon X \to X$ such that $im=m\langle i\pi_2, i\pi_1\rangle$. The triple $(\theta,\varphi,m)$, with $\theta=\langle i\pi_1,m\rangle$ and $\varphi=\langle m,i\pi_2\rangle$, is an involutive-2-link if and only if:
\begin{align*}
    m(i(x),m(x,y))=y\\
    m(m(x,y),i(y))=x
\end{align*}
for all $x,y\in X$.
\end{proposition}
\begin{proof}
    The two conditions above are equivalent to the requirement that $\theta$ and $\varphi$ are involutions. The interlinked condition $\theta\varphi\theta=\varphi\theta\varphi$ follows from the hypoteses $i(m(x,y))=m(i(y),i(x))$. 
\end{proof}

\section{The main result}

In this section we characterize those involutive-2-links which are groupoids, i.e., that are in the image of the functor described in Proposition \ref{thm: the functor}.

Let us first observe some necessary conditions. The notion of contractible pair in the sense of Beck can be found in \cite{MacLane}, p. 150.

\begin{proposition}\label{thm: prop contract}
If an involutive-2-link $(\theta,\varphi,m)$ is the image of an internal groupoid then the pairs $(m,m\theta)$ and $(m,m\varphi)$ are jointly mono\-morphic and contractible in the sense of Beck, i.e., there exist morphisms $e_1,e_2\colon C_1\to C_2$ such that \begin{align}
    me_1=1_{C_1},\quad m\theta e_1 m=m\theta e_1 m\theta\\
    me_2=1_{C_1},\quad m\varphi e_2 m=m\varphi e_2 m\varphi.
\end{align}
\end{proposition}
\begin{proof}
It suffices to take $e_1=\langle 1,ed\rangle$ and $e_2=\langle ec,1\rangle$, while observing that $(m,m\theta)$ and $(m,m\varphi)$ are respectively the kernel pairs of $d$ and $c$, the domain and codomain morphisms of the groupoid from which the triple $(\theta,\varphi,m)$ is obtained as in Proposition \ref{thm: the functor}.
\end{proof}

We say that a digraph (i.e. a pair of parallel morphisms) is  bi-exact  if when considered as a span it can be completed into a commutative square which is both a pullback and a pushout and moreover, if considered as a cospan, it can be completed into another commutative square which is both a pullback and a pushout. In other words, a digraph such as 
\begin{equation}\label{diag: di-graph pi_1 pi_2}
\xymatrix{C_2 \ar@<0.5ex>[r]^{\pi_2}\ar@<-0.5ex>[r]_{\pi_1} & C_1}
\end{equation}
is bi-exact precisely when the zig-zag 
\begin{equation}\label{diag: zig-zag}
\xymatrix{
& C_2\ar[d]^{\pi_1}\\ C_2 \ar[r]_{\pi_2}\ar[d]_{\pi_1} & C_1\\ C_1 }
\end{equation}
can be completed with two commutative squares 
\begin{equation}\label{diag: zig-zag completed}
\xymatrix{ C_3\ar@{-->}[r]^{p_2} \ar@{-->}[d]_{p_1}
& C_2\ar[d]^{\pi_1}\\ C_2 \ar[r]_{\pi_2}\ar[d]_{\pi_1} & C_1 \ar@{-->}[d]^{c}\\ C_1 \ar@{-->}[r]^{d} & C_0 }
\end{equation}
which  are both simultaneously a pullback and pushout.

\begin{proposition}
If an involutive-2-link $(\theta,\varphi,m)$ is the image of an internal groupoid then the pair $(m\varphi,m\theta)$ is bi-exact.
\end{proposition}
\begin{proof}
If $(\theta,\varphi,m)$ is obtained from an internal groupoid then $m\varphi=\pi_1$ and $m\theta=\pi_2$ which can be completed into commutative squares as displayed in $(\ref{diag: zig-zag completed})$. The two squares are pullbacks by hypotheses. The fact that the two squares are also pushouts  is an easy consequence of the fact that they are split squares. 
\end{proof}

The following result presents the desired characterization of those involutive-2-links which are the image of internal groupoids.

\begin{theorem}\label{thm: main}
An involutive-2-link $(\theta,\varphi,m\colon C_2\to C_1)$ is an internal groupoid if and only if the pairs $(m,m\theta)$ and $(m,m\varphi)$ are jointly monomorphic, there exist two morphisms $e_1,e_2\colon C_1\to C_2$ such that
\begin{align}
    me_1=1_{C_1}=me_2\label{eq: main 1}\\
    \theta e_2=e_2,\quad \varphi e_1=e_1\label{eq: main 2}\\
    m\theta\varphi e_2=m\varphi\theta e_1\label{eq: main 3 =i}\\
    m\theta e_1 m\varphi=m\varphi e_2 m\theta\label{eq: main 4 ed=ec}\\
    m\theta e_1 m=m\theta e_1 m\theta, \quad m\varphi e_2 m=m\varphi e_2 m\varphi,\label{eq: main 5 contract} 
\end{align}
the pair $(m\varphi,m\theta)$ is bi-exact (as illustrated in diagram (\ref{diag: zig-zag completed}) with $m\varphi$ as $\pi_1$ and  $m\theta$ as $\pi_2$), and the two induced morphisms $m_1,m_2\colon C_3\to C_2$, determined by
\begin{align*}
    \pi_1m_1=mp_1,\quad \pi_2 m_1=\pi_2 p_2\\
    \pi_1 m_2=\pi_1 p_1, \quad \pi_2 m_2=m p_2
\end{align*}
are such that $mm_1=mm_2$.
\end{theorem}
The proof is mainly technicall checking and it involves routine calculations, so we omit the details. It is clear that if $(\theta,\varphi,m)$ is obtained from an internal groupoid then all conditions are satisfyied. On the other hand, the fact that the pairs $(m,m\theta)$ and $(m,m\varphi)$ are jointly monomorphic uniquely determines the morphisms $e_1$ and $e_2$. The involution morphism $i\colon C_1\to C_1$ is obtained by condition (\ref{eq: main 3 =i}) either as $i=m\theta\varphi e_2$ or as $i=m\varphi\theta e_1$. The morphism $e\colon C_0\to C_1$ is uniquely determined by condition (\ref{eq: main 4 ed=ec}) as such that $ed=m\theta e_1$ and $ec=m\varphi e_2$. Conditions (\ref{eq: main 5 contract}) and (\ref{eq: main 1}) assert the contractibility of the pairs $(m,m\theta)$ and $(m,m\varphi)$ as in Proposition \ref{thm: prop contract}. The condition (\ref{eq: main 2}) is a central ingredient in the proof and it gives for example $e_1e=e_2e$. The two morpphisms $m_1$ and $m_2$ are well defined because $dm=d\pi_2$ and $cm=c\pi_1$ (with $m\varphi$ as $\pi_1$ and  $m\theta$ as $\pi_2$) which follows from the conditions in the Theorem.

\section{Conclusion}
With the categorical equivalence established by Theorem \ref{thm: main} and Pro\-position  \ref{thm: the functor} between internal groupoids and suitable involutive-2-links, it is now a straightforward task to extend it to the level of 2-cells. It is also tempting to explore the possibility of applying these results to internal categories rather than internal groupoids.

 Some advantages of considering an internal groupoid as an involutive-2-link are: (i) it can be worked out in arbitrary categories, even when pullbacks are not well-behaved as for smooth manifolds and the study of Lie groupoids; (ii) it suggests a straightforward generalization of n-dimensional groupoid as for Loday's cat-n-groups; (iii) it allows a systematic study of several properties exhibited by the fibrations and pre-fibrations used in descent theory and categorical Galois theory; (iv) it provides a convenient way of comparing internal groupoids in different ambient categories by investigating whether they are reduced for example to give a crossed module. In addition, several technical results on particular important cases such as pre-ordered groupoids or topological groupoids can be unified using the common language of involutive-2-links. Finally, the category of involutive-2-links certainly deserves to be studied by itself. 

 \section*{Acknowledgement}

 This work is supported by Fundação para a Ciência e a Tecnologia FCT/MCTES (PIDDAC) through the following Projects:  Associate Laboratory ARISE LA/P/\-0112/2020; UIDP/04044/2020; UIDB/04\-044/2020; PAMI - ROTEIRO/0328/2013 (Nº 022158); MATIS (CEN\-TRO-01-0145-FEDER-000014 - 3362); Generative.Thermodynamic; by CDRSP and ESTG from the Polytechnic Institute of Leiria.

\end{document}